\documentclass[11pt]{amsart}

\usepackage{amssymb}
\usepackage{mathtools}
\usepackage[numbers,sort&compress]{natbib}
\usepackage{xcolor}
\usepackage[margin=1.4in]{geometry}
\usepackage{hyperref}

\hypersetup{
  pdfauthor={Qian Ai, Feng Guo, Qi Zhou},
  pdftitle={A Functional Central Limit Theorem for Window Counts of Hardy--Szeg\H{o} Zeros},
  pdfsubject={Functional central limit theorem for window counts of Hardy--Szeg\H{o} zeros},
  pdfkeywords={Brillinger mixing, determinantal point process, functional central limit theorem, Hardy--Szeg\H{o} Gaussian analytic function, projected zero process}
}

\title[FCLT for Hardy--Szeg\H{o} zero counts]
{A Functional Central Limit Theorem for Window Counts of Hardy--Szeg\H{o} Zeros}

\author{Qian Ai}
\address{School of Mathematical Sciences, Soochow University,
Suzhou 215006, P. R. China}
\email{20234007001@stu.suda.edu.cn}

\author{Feng Guo}
\address{School of Mathematics, Nanjing University of Aeronautics and Astronautics,
Nanjing 210016, P. R. China}
\email{70207994@nuaa.edu.cn}

\author{Qi Zhou}
\thanks{Q. Z. was supported by National Natural Science Foundation of China
(No.~12501162), China Postdoctoral Science Foundation (No.~2024M762280),
and Natural Science Foundation of Jiangsu Province (No.~BK20250832).}
\address{School of Mathematical Sciences, Soochow University,
Suzhou 215006, P. R. China}
\email{zhouqi@suda.edu.cn}



\numberwithin{equation}{section}

\theoremstyle{plain}
\newtheorem{theorem}{Theorem}[section]
\newtheorem{lemma}[theorem]{Lemma}
\newtheorem{proposition}[theorem]{Proposition}
\newtheorem{corollary}[theorem]{Corollary}

\theoremstyle{definition}
\newtheorem{definition}[theorem]{Definition}
\newtheorem{example}[theorem]{Example}

\theoremstyle{remark}
\newtheorem{remark}[theorem]{Remark}

\begin{document}

\begin{abstract}
The Hardy--Szeg\H{o} zero process, investigated in the disk by Peres and Vir\'ag through the independent identically distributed Gaussian analytic function, is a canonical conformally invariant determinantal point process. This paper studies its upper half-plane realization. Although conformally equivalent to the disk model, this realization has its own natural geometry: real-translation invariance turns the process into a stationary object along the boundary and makes long horizontal windows the natural observables. For every admissible height window, we prove a Donsker-type functional central limit theorem for the centered zero counts in expanding horizontal windows, with an explicit intensity and variance depending on the height window. The proof is based on factorial cumulants and Brillinger mixing. The main technical input is a family of all-order integrability estimates for reduced cumulant densities, obtained by exploiting the determinantal cycle structure before integrating over the height variables. As further consequences, we derive an explicit covariance density, an asymptotic variance formula, and a macroscopic Gaussian white-noise limit for linear statistics.
\end{abstract}

\maketitle

\section{Introduction and main results}
\label{sec:introduction}

\noindent\textbf{The Hardy--Szeg\H{o} model and the upper-half-plane geometry.}
The Hardy--Szeg\H{o} zero process is a canonical conformally invariant
determinantal point process arising from zeros of Gaussian analytic functions.
Its disk realization was studied by Peres and Vir\'ag~\cite{PeresVirag2005}
as the zero set of the independent identically distributed Gaussian power
series on the unit disk.  In this paper we study the same conformally invariant
process in its upper-half-plane realization, which makes a different probabilistic geometry explicit: the process is
stationary under real translations.  This turns horizontal strips and long
boundary-parallel windows into natural observables, and leads to fluctuation
questions that are not transparent in disk coordinates.

Let
\(\mathbb H:=\{z\in\mathbb C:\operatorname{Im} z>0\},\)
and let \(\mathrm dA\) denote planar Lebesgue measure.  We consider the
determinantal point process on \((\mathbb H,\mathrm dA)\) with Hermitian kernel
\(K_{\mathbb H}(z,w)
        =
        -\frac{1}{\pi(z-\overline w)^2},
        ~ z,w\in\mathbb H.\)
Equivalently, this process is the zero set of the Hardy--Szeg\H{o} Gaussian
analytic function
\(f_{\mathbb H}(z)
        =
        (\varphi'(z))^{1/2}
        \sum_{n=0}^{\infty}\xi_n\varphi(z)^n,
        ~ z\in\mathbb H,\)
where \((\xi_n)_{n\geq0}\) are independent standard complex Gaussian random
variables, \(\varphi:\mathbb H\to\mathbb D\) is the Cayley map
\(\varphi(z):=\frac{z-i}{z+i},\)
and a holomorphic branch of \((\varphi')^{1/2}\) is fixed.  The explicit
kernel and the horizontal stationarity are the two structural features used
throughout the paper.  A more detailed account of the analytic function and
its zero process is given in Section~\ref{sec:model-projection}.

\medskip

\noindent\textbf{Long horizontal windows.}
The upper-half-plane realization suggests a natural macroscopic question:
how do the zeros fluctuate in a window whose horizontal length tends to
infinity while its height range remains fixed?  For a height window
\(I\subset(0,\infty)\), define
\[
        N_I(L)
        :=
        \#\{z\in\mathcal Z_{\mathbb H}:
        0\leq \operatorname{Re}z\leq L,\ 
        \operatorname{Im}z\in I\}.
\]
Thus \(N_I(L)\) counts zeros in the horizontal window \([0,L]\times I\).
Because the process is invariant under real translations, this is a stationary
one-dimensional counting problem along the boundary direction, but with
nontrivial dependence coming from the two-dimensional determinantal structure.

Our main result is a Donsker-type functional central limit theorem for these
window counts.  For every admissible height window \(I\), the centered counting
path
\(\bigl((N_I(Lt)-\lambda_I Lt)/L^{1/2}\bigr)_{0\leq t\leq1}\)
converges in distribution, as \(L\to\infty\), to Brownian motion with variance
parameter \(\sigma_I^2\).  The intensity \(\lambda_I\) and the variance
\(\sigma_I^2\) are given explicitly in terms of the height window.  In
particular, the result identifies both the order of the fluctuations and the
precise limiting variance for every admissible height range.

\medskip

\noindent\textbf{Method and further consequences.}
The proof is based on factorial cumulants and Brillinger mixing.  The main
technical step is to prove all-order integrability estimates for the reduced
cumulant densities associated with the zero process.  We do this by exploiting
the determinantal cycle structure before integrating over the height variables.
This order of analysis is important: the cancellations and decay needed for
summability are most transparent at the level of the determinantal cycles.

The same estimates also yield several quantitative consequences.  We derive an
explicit covariance density for the horizontal counting process, an asymptotic
variance formula for long windows, and a macroscopic Gaussian white-noise limit
for linear statistics.  These results show that the upper-half-plane
realization is not only a convenient coordinate representation of the disk
model, but also a natural setting in which stationary long-window asymptotics
can be formulated and proved.

\medskip

\noindent\textbf{Related work.}
The Hardy--Szeg\H{o} zero process belongs to the general theory of Gaussian
analytic functions and their zero sets; see the monograph of Hough, Krishnapur, Peres and Vir\'ag~\cite{HoughKrishnapurPeresVirag2009}.
Stationary Gaussian analytic function models and their fluctuation theory
provide a particularly close point of comparison; see, for example,
Feldheim~\cite{Feldheim2013}. The process is also closely connected to the
theory of determinantal point processes and to fluctuation problems for random
zeros. Related developments include work on correlation asymptotics,
equidistribution, central limit theorems, and rigidity-type phenomena for
random analytic zeros; see, for example,
\cite{ShiffmanZelditch1999,BleherShiffmanZelditch2000,Sodin2000,
SodinTsirelson2004,NazarovSodinVolberg2008,Shiffman2008,
Krishnapur2009,NazarovSodin2011}. This broader line of work is rooted in the
classical study of zeros of random polynomials; see, for example,
\cite{Hammersley1956,Friedman1990,BogomolnyBohigasLeboeuf1992,Kostlan1993,EdelmanKostlan1995}.

\medskip

\noindent\textbf{Determinantal structure and projection.}
There is also a second motivation for studying horizontal windows.  They arise
from a natural projection of a planar determinantal point process.  For a height
window \(I\subset(0,\infty)\), let
\(
\mathcal X_I
:=
\sum_{z\in\mathcal Z_{\mathbb H}:\operatorname{Im}z\in I}
\delta_{\operatorname{Re}z}
\)
be the horizontal projection of the zeros whose heights lie in \(I\).  Then
\(N_I(L)=\mathcal X_I([0,L]).\)
Thus the long-window problem may be viewed as a fluctuation problem for the
one-dimensional projected process \(\mathcal X_I\).

This viewpoint connects the present work with a general issue in the theory of
determinantal point processes.  Determinantal processes are powerful because
their correlation functions have explicit determinant forms and many
probabilistic quantities can be analyzed through operator identities; see, for
example,
\cite{Macchi1975,Soshnikov2000,Lyons2003,ShiraiTakahashi2003,
Borodin2011,LavancierMollerRubak2015}.  This structure, however, is not
automatically preserved under geometric operations.  Projection is particularly
delicate, since it collapses one coordinate and may change the local structure
of the process.  The projected process \(\mathcal X_I\) therefore provides a
concrete setting in which one can ask how much of the planar determinantal
structure remains usable after a natural one-dimensional observation.

For the Hardy--Szeg\H{o} zero process this loss of structure is genuine.
Although \(\mathcal Z_{\mathbb H}\) is determinantal in the plane,
Proposition~\ref{prop:not-translation-invariant-dpp} shows that
\(\mathcal X_I\) does not admit a Hermitian translation-invariant locally
trace-class determinantal representation on \(\mathbb R\).  Consequently, the
standard determinantal central limit theorems of
Costin and Lebowitz~\cite{CostinLebowitz1995} and
Soshnikov~\cite{Soshnikov2000Gaussian,Soshnikov2002} do not apply directly.
For related fluctuation results in planar determinantal settings, see
\cite{RiderVirag2007H1Noise,RiderVirag2007CircularLaw,
AmeurHedenmalmMakarov2011}.  Our strategy is instead to work before projection:
we use the planar determinantal cycle expansion first, and only afterwards
integrate over the height variables.

\subsection{Main results}
\label{subsec:main-results}

We first specify the class of height windows.  The condition below guarantees
that the projected process has finite intensity on bounded intervals and
controls the growth of the planar zero density near \(\partial\mathbb H\).

\begin{definition}
A height window \(I\subset(0,\infty)\), not necessarily an interval, is called admissible if \(\int_I y^{-2}\,\mathrm d y<\infty.\)
\end{definition}

Our first result is the all-order cumulant estimate that drives the paper.
We use the standard terminology of reduced factorial cumulants and Brillinger
mixing for stationary point processes; the relevant conventions are recalled
in Appendix~\ref{app:cumulants}.  Brillinger mixing, namely the finite total
variation of all reduced factorial cumulant measures of order at least two,
is a useful framework for proving Gaussian limits for dependent point
processes through cumulant bounds; see, for example,
\cite{Brillinger1972,Ivanoff1982,Janson1988,DaleyVereJones2003,DaleyVereJones2008,BiscioLavancier2016}.  In the present setting, the estimate below
is the substitute for the unavailable one-dimensional determinantal structure.

\begin{theorem}
\label{thm:brillinger-mixing}
Assume that \(I\subset(0,\infty)\) is admissible and has positive Lebesgue measure. Then \(\mathcal{X}_I\) is Brillinger mixing. More precisely, for every integer \(k\geq2\), the \(k\)-th reduced factorial cumulant measure of \(\mathcal{X}_I\) is absolutely continuous with respect to Lebesgue measure on \(\mathbb{R}^{k-1}\), with density \(c_I^{(k)}\in L^1(\mathbb{R}^{k-1})\).
Moreover,
\[
\left\|c_I^{(k)}\right\|_{L^1(\mathbb{R}^{k-1})}
\leq
\frac{(k-1)!}{2^{k+1}k\pi}
\left(
\int_I y^{-(1+1/k)}\,\mathrm{d}y
\right)^k.
\]
\end{theorem}

This estimate converts the planar determinantal cycle expansion into
one-dimensional mixing information for the horizontal process.  The next result
identifies the second-order quantities that determine the limiting Gaussian
variance.

\begin{theorem}
\label{thm:variance-formula}
Assume that \(I\subset(0,\infty)\) is admissible and has positive Lebesgue measure. Then \(\mathcal{X}_I\) has intensity
\(\lambda_I:=\int_I (4\pi y^2)^{-1}\,\mathrm{d}y\), and reduced covariance density
\[
g_I(r)
=
-\frac{1}{\pi^2}
\int_I\int_I
\frac{1}{\left(r^2+(y_1+y_2)^2\right)^2}
\,\mathrm{d}y_1\,\mathrm{d}y_2,
\qquad r\in\mathbb{R}.
\]
Moreover, \(g_I\in L^1(\mathbb{R})\), and
\[
\lim_{L\to\infty}
\frac{\operatorname{Var}(N_I(L))}{L}
=
\sigma_I^2
:=
\lambda_I+\int_{\mathbb{R}}g_I(r)\,\mathrm{d}r
>0.
\]
\end{theorem}

\begin{example}
\label{ex:interval-window}
Let \(I=[a,b]\), \(0<a<b<\infty\). Then
\[
\lambda_{I}
=
\frac{b-a}{4\pi ab},
\qquad
\sigma_{I}^{2}
=
\frac{(b-a)(3a+b)}{8\pi ab(a+b)}
>0.
\]
In particular, if \(b\) is fixed and \(a\downarrow0\), then
\(\lambda_{[a,b]}\sim 1/(4\pi a)\) and \(\sigma_{[a,b]}^2\sim 1/(8\pi a)\), reflecting the growth of the planar zero density near \(\partial\mathbb H\).
\end{example}

These constants determine the macroscopic Gaussian limits.
For \(L\geq1\) and \(\varphi\in C_c^\infty(\mathbb{R})\), define
\[
\Xi_L(\varphi)
:=
\frac{1}{\sqrt{L}}
\left(
\int_{\mathbb{R}}\varphi\left(\frac{x}{L}\right)\mathcal{X}_I(\mathrm{d}x)
-
\lambda_I L\int_{\mathbb{R}}\varphi(u)\,\mathrm{d}u
\right).
\]

\begin{theorem}
\label{thm:white-noise-limit}
Assume that \(I\subset(0,\infty)\) is admissible and has positive Lebesgue measure. As \(L\to\infty\), for every integer \(m\geq1\) and every \(\varphi_1,\ldots,\varphi_m\in C_c^\infty(\mathbb{R})\), the random vector
\((\Xi_L(\varphi_1),\ldots,\Xi_L(\varphi_m))\)
converges in distribution to a centered Gaussian vector
\((\Xi(\varphi_1),\ldots,\Xi(\varphi_m))\)
with covariance
\[
\mathbb{E}[\Xi(\varphi)\Xi(\psi)]
=
\sigma_I^2
\int_{\mathbb{R}}\varphi(u)\psi(u)\,\mathrm{d}u.
\]
\end{theorem}

The main path-level result is the following Donsker-type functional central
limit theorem, in the sense of the invariance-principle scaling introduced by
Donsker~\cite{Donsker1951}. Define
\(B_L(t)
:=
\frac{N_I(Lt)-\lambda_I Lt}{\sqrt{L}},
~ 0\leq t\leq1.\)

\begin{theorem}
\label{thm:fclt}
Assume that \(I\subset(0,\infty)\) is admissible and has positive Lebesgue measure. As \(L\to\infty\), \(B_L\) converges in distribution in \(D([0,1])\), equipped with the Skorokhod \(J_1\) topology, to Brownian motion \(B=(B(t))_{t\in[0,1]}\) with covariance
\(\mathbb{E}[B(s)B(t)]=\sigma_I^2\min\{s,t\}.\)
\end{theorem}

Evaluating the limiting path at \(t=1\) gives the scalar central limit theorem for interval counts.

\begin{corollary}
\label{cor:interval-count-clt}
Assume that \(I\subset(0,\infty)\) is admissible and has positive Lebesgue measure. Then, as \(L\to\infty\),
\((N_I(L)-\lambda_I L)/\sqrt{L}\Rightarrow \mathcal{N}(0,\sigma_I^2).\)
\end{corollary}

\subsection{Organization of the paper.}
Section~\ref{sec:model-projection} defines the Hardy--Szeg\H{o} Gaussian analytic function on \(\mathbb{H}\), recalls the determinantal zero process, introduces the projected strip process, and proves its basic structural properties. Section~\ref{sec:correlations-variance} derives the projected correlation and cumulant formulas, identifies the covariance density and variance asymptotics, and proves the small-gap obstruction to one-dimensional Hermitian translation-invariant determinantal representations. Section~\ref{sec:brillinger-mixing} proves the cycle convolution identity, the all-order \(L^1\)-bound for reduced factorial cumulant densities, Brillinger mixing, and the resulting linear cumulant estimates. Section~\ref{sec:gaussian-limits} proves the macroscopic white-noise limit and the functional central limit theorem. Appendices~\ref{app:cumulants} and~\ref{app:tightness} contain, respectively, the point-process cumulant conventions and the tightness tools used in the proof.

\section{The Hardy--Szeg\H{o} zero process and its projection}
\label{sec:model-projection}

In this section we record the Hardy--Szeg\H{o} covariance kernel, the determinantal structure of the zero process, and the basic structural properties of the projected process used later.

\subsection{The Hardy--Szeg\H{o} Gaussian analytic function}
\label{subsec:hardy-szego-gaf}

Recall that the Hardy--Szeg\H{o} Gaussian analytic function on \(\mathbb{H}\) is
\[
f_{\mathbb{H}}(z)
=
(\varphi'(z))^{1/2}
\sum_{n=0}^{\infty}\xi_{n}\varphi(z)^{n}, \qquad z\in\mathbb{H},
\]
where \(\varphi\) is the Cayley map, the branch of \((\varphi')^{1/2}\) is the one fixed above, and \((\xi_n)_{n\geq0}\) are independent standard complex Gaussian random variables. The series converges almost surely locally uniformly on \(\mathbb{H}\), because \(\phi\) maps compact subsets of \(\mathbb{H}\) into compact subsets of \(\mathbb{D}\), and the usual Gaussian power series
\(f_{\mathbb{D}}(u):=\sum_{n=0}^{\infty}\xi_{n}u^{n},
 u\in\mathbb{D},\)
converges almost surely locally uniformly on \(\mathbb{D}\). Thus \(f_{\mathbb{H}}\) is almost surely holomorphic on \(\mathbb{H}\).

The covariance function of \(f_{\mathbb H}\) is \(\mathbb E 
f_{\mathbb H}(z)\overline{f_{\mathbb H}(w)}
=
\frac{i}{z-\overline w}\),
\(z,w\in\mathbb H.\)
Since this covariance kernel is invariant under simultaneous horizontal translations, and \(f_{\mathbb H}\) is centered complex Gaussian, the law of \(f_{\mathbb H}\) is horizontally translation-invariant.

\begin{remark}
\label{rem:square-root-factor}
The factor \((\varphi')^{1/2}\) is the usual conformal covariance factor for the Hardy--Szeg\H{o} kernel; see, for example, \cite{Duren1970,Bell1992} for the classical Hardy-space and conformal mapping background behind this transformation rule. It does not change the zero set, because it is holomorphic and nowhere zero on \(\mathbb{H}\). Its purpose is to transport the Hardy--Szeg\H{o} covariance on the disk to the upper half-plane normalization.
\end{remark}

\subsection{The determinantal zero process on the upper half-plane}
\label{subsec:zero-dpp}

We view \(\mathcal Z_{\mathbb H}\) as a random counting measure on \(\mathbb H\). Since \(f_{\mathbb H}\) is a nonzero Gaussian analytic function, its zeros are almost surely locally finite and, by the standard nondegeneracy theorem for Gaussian analytic functions \cite[Chapter~2]{HoughKrishnapurPeresVirag2009}, simple as planar zeros. Moreover, from Peres and Vir\'ag~\cite[Theorem~1]{PeresVirag2005}, conformal covariance of zeros, and the transformation rule for correlation functions under the Cayley map, the zero process \(\mathcal Z_{\mathbb H}\) is determinantal on \((\mathbb H,\mathrm dA)\), with correlation kernel \(K_{\mathbb H}\). In particular, the one-point intensity of \(\mathcal Z_{\mathbb H}\) is
\begin{equation}
\label{eq:upper-half-plane-one-point-intensity}
\rho_{\mathbb H}^{(1)}(x+iy)
=
\frac{1}{4\pi y^2},
\qquad x\in\mathbb R,\ y>0.
\end{equation}
The zero process is stationary under horizontal translations: for \(a\in\mathbb R\), \(\mathcal Z_{\mathbb H}-a:=\{z-a:z\in\mathcal Z_{\mathbb H}\}
\stackrel{\mathrm d}=
\mathcal Z_{\mathbb H}\).

\subsection{Basic properties of the projected process}
\label{subsec:basic-properties-projected}

Recall that \(\mathcal{X}_{I}\) is obtained by restricting
\(\mathcal{Z}_{\mathbb{H}}\) to the strip
\(
S_I:=\{x+iy\in\mathbb H:x\in\mathbb R,\ y\in I\}
\)
and projecting the remaining zeros onto the real axis.

\begin{lemma}
\label{lem:local-finiteness-intensity}
Let \(I\subset(0,\infty)\) be admissible. Then \(\mathcal{X}_{I}\) is locally finite and, for every bounded Borel set \(B\subset\mathbb{R}\),
\[
\mathbb{E}[\mathcal{X}_{I}(B)] = \lambda_{I}|B|,
\qquad
\lambda_{I}:=\int_{I}\frac{1}{4\pi y^{2}}\,\mathrm{d}y<\infty.
\]
\end{lemma}

\begin{proof}
For bounded Borel \(B\subset\mathbb{R}\), set \(R_{B,I}:=\{x+iy\in\mathbb{H}:x\in B,\ y\in I\}\). By \eqref{eq:upper-half-plane-one-point-intensity},
\[
\mathbb{E}[\mathcal{X}_{I}(B)]
=\int_{R_{B,I}}\rho_{\mathbb{H}}^{(1)}(z)\,\mathrm{d}A(z)
=\int_{B}\int_{I}\frac{1}{4\pi y^{2}}\,\mathrm{d}y\,\mathrm{d}x
=\lambda_{I}|B|.
\]
The integral is finite by admissibility, and a nonnegative integer-valued random variable with finite expectation is finite almost surely.
\end{proof}

\begin{lemma}
\label{lem:stationarity-projected}
Let \(I\subset(0,\infty)\) be admissible. Then \(\mathcal{X}_{I}\) is stationary: for every \(a\in\mathbb{R}\), \(m\geq1\), and bounded Borel sets \(B_{1},\ldots,B_{m}\subset\mathbb{R}\), 
\(\left(
\mathcal{X}_{I}(B_{1}+a),\ldots,\mathcal{X}_{I}(B_{m}+a)
\right)
\stackrel{\mathrm{d}}{=}
\left(
\mathcal{X}_{I}(B_{1}),\ldots,\mathcal{X}_{I}(B_{m})
\right).\)
\end{lemma}

\begin{proof}
This follows immediately from the stationarity of the zero process \(\mathcal Z_{\mathbb H}\).
\end{proof}

\begin{lemma}
\label{lem:simplicity-projected}
Let \(I\subset(0,\infty)\) be admissible. Then \(\mathcal{X}_{I}\) is simple: almost surely,
\(\mathcal{X}_{I}(\{x\})\in\{0,1\}\) for every \(x\in\mathbb{R}\).
\end{lemma}

\begin{proof}
It suffices to rule out multiple projected points on each bounded interval
\(J\subset\mathbb R\). For \(\varepsilon>0\), set
\[
M_{J,\varepsilon}
:=
\sum_{\substack{z_{1},z_{2}\in\mathcal{Z}_{\mathbb{H}}\cap S_{I}\\ z_{1}\neq z_{2}}}
\mathbf{1}_{J}(\operatorname{Re}z_{1})
\mathbf{1}_{\{|\operatorname{Re}z_{2}-\operatorname{Re}z_{1}|<\varepsilon\}} .
\]
If two points of \(\mathcal Z_{\mathbb H}\cap S_I\) have the same real part in \(J\), then
\(M_{J,\varepsilon}\ge1\) for every \(\varepsilon>0\). Hence, by Markov's inequality and the
factorial moment formula,
\[
\mathbb P\bigl(\exists x\in J:\mathcal X_I(\{x\})\ge2\bigr)
\le \mathbb E M_{J,\varepsilon}=
\int_J\int_{|x'-x|<\varepsilon}\int_I\int_I
\rho_{\mathbb H}^{(2)}(x+iy,x'+iy')\,
\mathrm dy\,\mathrm dy'\,\mathrm dx'\,\mathrm dx .
\]
Since the process is determinantal with Hermitian kernel,
\[
\rho_{\mathbb H}^{(2)}(z,w)
=K_{\mathbb H}(z,z)K_{\mathbb H}(w,w)-|K_{\mathbb H}(z,w)|^2
\le K_{\mathbb H}(z,z)K_{\mathbb H}(w,w),
\]
and therefore
\[
\mathbb E M_{J,\varepsilon}
\le
\int_J\int_{|x'-x|<\varepsilon}
\left(\int_I\frac{\mathrm dy}{4\pi y^2}\right)^2
\mathrm dx'\,\mathrm dx
=2\varepsilon |J|\lambda_I^2 .
\]
Letting \(\varepsilon\downarrow0\) gives the claim on \(J\), and taking
\(J=[-n,n]\), \(n\ge1\), completes the proof.
\end{proof}

\begin{remark}
\label{rem:projection-does-not-create-collisions}
Lemma~\ref{lem:simplicity-projected} does not follow from simplicity of the planar zero process, since \(z\mapsto\operatorname{Re}z\) may identify distinct points. The proof uses the two-point intensity of the planar determinantal process and the admissibility condition \(\int_{I}y^{-2}\,\mathrm{d}y<\infty\) to show that the expected number of near-colliding projected pairs in a bounded interval is \(O(\varepsilon)\), ruling out exact projected collisions.
\end{remark}

\section{Projected correlations and second-order structure}
\label{sec:correlations-variance}

In this section we derive the correlation and cumulant structure of the projected strip process. Since \(\mathcal{Z}_{\mathbb{H}}\) is determinantal before projection, the corresponding densities of \(\mathcal{X}_{I}\) are obtained by integrating the planar determinantal formulas over the height variables. The second-order formulas yield the variance asymptotics and rule out a Hermitian translation-invariant determinantal representation on \(\mathbb{R}\).

\subsection{Projected \(k\)-point intensities}
\label{subsec:projected-intensities}

The projected intensities are obtained by integrating the planar intensities over the height variables.

\begin{proposition}
\label{prop:projected-intensities}
Let \(I\subset(0,\infty)\) be admissible and \(k\geq1\). Then \(\mathcal{X}_{I}\) admits \(k\)-point intensity functions, and for pairwise distinct \(x_{1},\ldots,x_{k}\in\mathbb{R}\),
\begin{equation}
\label{eq:projected-k-point-intensity}
\rho_{I}^{(k)}(x_{1},\ldots,x_{k})
=
\int_{I^{k}}
\det\left(
K_{\mathbb{H}}(x_{i}+iy_{i},x_{j}+iy_{j})
\right)_{i,j=1}^{k}
\,\mathrm{d}y_{1}\cdots\mathrm{d}y_{k}.
\end{equation}
\end{proposition}

\begin{proof}
Let \(F:\mathbb{R}^{k}\to[0,\infty)\) be measurable with compact support. Since \(\mathcal Z_{\mathbb H}\) is determinantal with kernel \(K_{\mathbb H}\), the factorial moment formula (Definition \ref{def:appendix-correlation-functions}) gives
\[
\begin{aligned}
&\mathbb E\sum_{x_{1},\ldots,x_{k}\in\mathcal X_I}^{\neq}
F(x_{1},\ldots,x_{k})\\
&\quad=
\int_{\mathbb R^k} F(x_{1},\ldots,x_{k})
\int_{I^k}
\det\left(
K_{\mathbb H}(x_i+iy_i,x_j+iy_j)
\right)_{i,j=1}^k
\,\mathrm dy_{1}\cdots\mathrm dy_k\,
\mathrm dx_{1}\cdots\mathrm dx_k .
\end{aligned}
\]
It remains only to note that the inner integral is finite. Indeed, for
\(z_j=x_j+iy_j\), 
\[
0\le
\det\left(K_{\mathbb H}(z_i,z_j)\right)_{i,j=1}^k
\le
\prod_{j=1}^k K_{\mathbb H}(z_j,z_j)
=
\prod_{j=1}^k \frac{1}{4\pi y_j^2},
\]
whose integral over \(I^k\) is \(\lambda_I^k<\infty\). Thus
\eqref{eq:projected-k-point-intensity} defines the \(k\)-point intensity.
\end{proof}

\begin{corollary}
\label{cor:projected-intensity-basic-properties}
Let \(I\subset(0,\infty)\) be admissible. Then \(\rho_{I}^{(1)}\equiv\lambda_I\), \(\rho_{I}^{(2)}(x_{1},x_{2})\) depends only on \(x_{2}-x_{1}\), and for every \(k\geq1\),
\(0\leq\rho_I^{(k)}(x_1,\ldots,x_k)\leq\lambda_I^k.\)
\end{corollary}

\subsection{Projected factorial cumulant densities}
\label{subsec:projected-factorial-cumulants}

We pass from projected correlations to factorial cumulant densities. The finite partition expansion is compatible with the height integrations in Proposition~\ref{prop:projected-intensities}.

\begin{proposition}
\label{prop:projected-factorial-cumulant-density}
Let \(I\subset(0,\infty)\) be admissible and \(k\geq2\). The \(k\)-th factorial cumulant measure of \(\mathcal{X}_{I}\) is absolutely continuous with density \(u_{I}^{(k)}\), and for pairwise distinct \(x_{1},\ldots,x_{k}\in\mathbb{R}\),
\begin{equation}
\label{eq:projected-factorial-cumulant-density}
u_{I}^{(k)}(x_{1},\ldots,x_{k})
=
\int_{I^{k}}
u_{\mathbb{H}}^{(k)}(x_{1}+iy_{1},\ldots,x_{k}+iy_{k})
\,\mathrm{d}y_{1}\cdots\mathrm{d}y_{k},
\end{equation}
where \(u_{\mathbb{H}}^{(k)}\) is the \(k\)-th factorial cumulant density of \(\mathcal{Z}_{\mathbb{H}}\).
\end{proposition}

\begin{proof} 
Substituting \eqref{eq:projected-k-point-intensity} into the partition formula in Definition~\ref{def:appendix-factorial-cumulant-densities}, and using the disjointness of the blocks, the finite sum may be moved inside the integral: 
\[ 
\begin{aligned} u_I^{(k)}(x_1,\ldots,x_k) &= \int_{I^k} \sum_{\pi\in\mathcal P([k])} (-1)^{|\pi|-1}(|\pi|-1)! \prod_{B\in\pi} \rho_{\mathbb H}^{(|B|)} \bigl((x_j+iy_j)_{j\in B}\bigr) \,\mathrm dy_1\cdots\mathrm dy_k \\ &= \int_{I^k} u_{\mathbb H}^{(k)}(x_1+iy_1,\ldots,x_k+iy_k) \,\mathrm dy_1\cdots\mathrm dy_k . \end{aligned} 
\] 
\end{proof}

\begin{corollary}
\label{cor:projected-reduced-factorial-cumulant-density}
Let \(I\subset(0,\infty)\) be admissible and \(k\geq2\). The \(k\)-th reduced factorial cumulant density of \(\mathcal{X}_{I}\) may be represented, off the reduced diagonals, by
\begin{equation}
\label{eq:reduced-cumulant-density-definition}
c_{I}^{(k)}(t_{2},\ldots,t_{k})
:=
u_{I}^{(k)}(0,t_{2},\ldots,t_{k}).
\end{equation}
For every bounded measurable compactly supported \(\varphi:\mathbb{R}\to\mathbb{R}\), if the integral on the right-hand side is absolutely convergent, then
\begin{equation}
\label{eq:factorial-cumulant-linear-statistic-reduced}
\operatorname{cum}_{k}^{!}
\left(
\int_{\mathbb{R}}\varphi(x)\,\mathcal{X}_{I}(\mathrm{d}x)
\right)
=
\int_{\mathbb{R}^{k}}
\varphi(x_{1})\cdots\varphi(x_{k})
c_{I}^{(k)}(x_{2}-x_{1},\ldots,x_{k}-x_{1})
\,\mathrm{d}x_{1}\cdots\mathrm{d}x_{k}.
\end{equation}
\end{corollary}

\begin{proof}
By Lemma~\ref{lem:stationarity-projected}, \(\mathcal{X}_{I}\) is stationary, so its factorial cumulant measures are translation invariant. Proposition~\ref{prop:projected-factorial-cumulant-density} gives the full cumulant density \(u_I^{(k)}\), hence the reduced density is represented off the reduced diagonals by \eqref{eq:reduced-cumulant-density-definition}. Formula \eqref{eq:factorial-cumulant-linear-statistic-reduced} follows from Lemma~\ref{lem:appendix-reduced-cumulant-linear-statistic}.
\end{proof}

The projected cumulant formula and the planar determinantal cycle expansion give the explicit representation used in the Brillinger-mixing estimate.

\begin{proposition}
\label{prop:projected-cycle-expansion}
Let \(I\subset(0,\infty)\) be admissible, \(k\geq2\), and let
\(\operatorname{Cyc}_{k}\) be the set of all \(k\)-cycles on \(\{1,\ldots,k\}\).
For pairwise distinct \(x_{1},\ldots,x_{k}\in\mathbb{R}\),
\begin{equation}
\label{eq:projected-cycle-expansion}
u_{I}^{(k)}(x_{1},\ldots,x_{k})
=
(-1)^{k-1}
\sum_{\sigma\in\operatorname{Cyc}_{k}}
\int_{I^{k}}
\prod_{j=1}^{k}
K_{\mathbb{H}}(x_{j}+iy_{j},x_{\sigma(j)}+iy_{\sigma(j)})
\,\mathrm{d}y_{1}\cdots\mathrm{d}y_{k}.
\end{equation}
\end{proposition}

\begin{proof}
This follows by substituting the determinantal cycle expansion of Lemma~\ref{lem:determinantal-cycle-expansion} into \eqref{eq:projected-factorial-cumulant-density} and moving the finite sum over \(\operatorname{Cyc}_{k}\) outside the height integral.
\end{proof}

\subsection{Second-order formulas and an interval-window example}
\label{subsec:second-order-formulas}

We record the second-order quantities of the projected process, starting with the reduced covariance density. Since \(\mathcal{X}_{I}\) is stationary, write \(\rho_{I}^{(2)}(r):=\rho_{I}^{(2)}(0,r)\) for \(r\neq0\). Diagonal values of factorial densities are irrelevant; the following formula defines the continuous representative used at \(r=0\).

\begin{proposition}
\label{prop:covariance-density}
Let \(I\subset(0,\infty)\) be admissible. The reduced covariance density has the
continuous representative
\begin{equation}
\label{eq:covariance-density-formula}
g_I(r)
=
-\frac{1}{\pi^2}
\int_I\int_I
\frac{1}{\bigl(r^2+(y_1+y_2)^2\bigr)^2}
\,\mathrm dy_1\,\mathrm dy_2,
\qquad r\in\mathbb R .
\end{equation}
Moreover, \(g_I\) is even and belongs to \(L^1(\mathbb R)\).
\end{proposition}

\begin{proof}
For \(r\neq0\), Proposition~\ref{prop:projected-intensities} with \(k=2\) and
\(\rho_I^{(1)}\equiv\lambda_I\) give
\[
\rho_I^{(2)}(r)-\lambda_I^2
=
-\int_I\int_I |K_{\mathbb H}(iy_1,r+iy_2)|^2
\,\mathrm dy_1\,\mathrm dy_2,
\]
which is exactly \eqref{eq:covariance-density-formula}. The same formula is
finite at \(r=0\), and dominated convergence gives continuity, by admissibility.
Evenness is immediate. Finally,
\[
\int_{\mathbb R}|g_I(r)|\,\mathrm dr
=
\frac{1}{2\pi}
\int_I\int_I \frac{1}{(y_1+y_2)^3}
\,\mathrm dy_1\,\mathrm dy_2
<\infty,
\]
again by admissibility. Hence \(g_I\in L^1(\mathbb R)\).
\end{proof}

\begin{remark}
\label{rem:covariance-density-decay}
If \(|I|<\infty\), then \eqref{eq:covariance-density-formula} gives the pointwise estimate
\(|g_I(r)|\leq |I|^2/(\pi^2 r^4)\), \(r\neq0.\)
For infinite admissible windows, such as \(I=(a,\infty)\) with \(a>0\), \(g_I\) is still integrable by Proposition~\ref{prop:covariance-density}, but the pointwise decay may be slower than \(r^{-4}\). This \(L^{1}\)-integrability is the input needed for the variance formula and Gaussian limits.
\end{remark}

\begin{proposition}
\label{prop:variance-formula}
Let \(I\subset(0,\infty)\) be admissible and set
\(\sigma_I^2:=\lambda_I+\int_{\mathbb R}g_I(r)\,\mathrm dr\). Then
\begin{equation}
\label{eq:variance-asymptotic}
\lim_{L\to\infty}\frac{1}{L}\operatorname{Var}(N_I(L))
=
\sigma_I^2 .
\end{equation}
Moreover, if \(|I|>0\), then \(\sigma_I^2>0\).
\end{proposition}

\begin{proof}
By the variance identity for simple point processes and
Proposition~\ref{prop:covariance-density},
\[
\frac{1}{L}\operatorname{Var}(N_I(L))
=
\lambda_I+\int_{\mathbb R}(1-|r|/L)_+g_I(r)\,\mathrm dr .
\]
Since \(g_I\in L^1(\mathbb R)\), dominated convergence gives
\eqref{eq:variance-asymptotic}.

By \eqref{eq:covariance-density-formula} and Tonelli's theorem,
\[
\sigma_I^2
=
\frac{1}{4\pi}\int_I \frac{\mathrm dy}{y^2}
-
\frac{1}{2\pi}\int_I\int_I \frac{\mathrm dy_1\,\mathrm dy_2}{(y_1+y_2)^3}.
\]
For each \(y_1>0\), the inner integral over \(I\) is bounded by its integral over
\((0,\infty)\), namely \(1/(2y_1^2)\), and the inequality is strict since
\((0,\infty)\setminus I\) has positive measure. Integrating over \(y_1\in I\) and
using \(|I|>0\) yields \(\sigma_I^2>0\).
\end{proof}

\begin{proof}[Proof of Theorem \ref{thm:variance-formula}]
The intensity formula is Lemma \ref{lem:local-finiteness-intensity}. The covariance density formula and its \(L^{1}\)-integrability are Proposition \ref{prop:covariance-density}. The existence, formula, and positivity of the asymptotic variance are Proposition \ref{prop:variance-formula}.
\end{proof}

\subsection{Small gaps and a determinantal obstruction}
\label{subsec:failure-translation-invariant-dpp}

Although \(\mathcal{Z}_{\mathbb{H}}\) is determinantal before projection, the projected process does not have the small-gap behavior of a Hermitian translation-invariant locally trace-class determinantal process on \(\mathbb{R}\). We now make this precise.

\begin{proposition}
\label{prop:small-gap-limit}
Let \(I\subset(0,\infty)\) be admissible with \(|I|>0\). Then the stationary
representative \(\rho_I^{(2)}(r):=\rho_I^{(2)}(0,r)\) satisfies
\begin{equation}
\label{eq:small-gap-limit}
\lim_{r\to0}\rho_I^{(2)}(r)
=
\int_I\int_I
\left(
\frac{1}{16\pi^2y_1^2y_2^2}
-
\frac{1}{\pi^2(y_1+y_2)^4}
\right)
\,\mathrm dy_1\,\mathrm dy_2
>0 .
\end{equation}
\end{proposition}

\begin{proof}
Since \(\rho_I^{(2)}(r)=\lambda_I^2+g_I(r)\), Proposition~\ref{prop:covariance-density}
gives
\[
\rho_I^{(2)}(r)
=
\int_I\int_I
\left[
\frac{1}{16\pi^2y_1^2y_2^2}
-
\frac{1}{\pi^2\bigl(r^2+(y_1+y_2)^2\bigr)^2}
\right]
\,\mathrm dy_1\,\mathrm dy_2 .
\]
Dominated convergence yields the limit in \eqref{eq:small-gap-limit}. The limiting
integrand is nonnegative by \((y_1+y_2)^2\ge4y_1y_2\), and is strictly positive
off the diagonal \(y_1=y_2\). Since \(|I|>0\), \(I^2\) has positive measure while
its diagonal has measure zero; hence the limiting integral is strictly positive.
\end{proof}

We now recall the small-gap behavior forced by a Hermitian translation-invariant determinantal structure on the line.

\begin{lemma}
\label{lem:translation-invariant-dpp-small-gap}
Let \(Y\) be a stationary determinantal point process on \(\mathbb{R}\) with finite
intensity \(\lambda\), admitting a Hermitian translation-invariant locally
trace-class kernel \(K_{\mathbb{R}}\) with respect to Lebesgue measure. Then, after
modifying the kernel on a null set, \(K_{\mathbb{R}}(x,y)=\mathcal{K}(x-y)\) for
a continuous function \(\mathcal{K}:\mathbb{R}\to\mathbb{C}\) with
\(\mathcal{K}(0)=\lambda\), and
\(\rho_{Y}^{(2)}(r)=\lambda^{2}-|\mathcal{K}(r)|^{2}\) for \(r\neq0\). Consequently,
\(\lim_{r\to0}\rho_{Y}^{(2)}(r)=0\).
\end{lemma}

\begin{proof}
By translation invariance and the Fourier representation of such kernels on
\(\mathbb R\), \(K_{\mathbb R}\) has multiplier \(q\); see
\cite[Section~3]{Soshnikov2000}. The Hermitian DPP existence criterion
\cite[Theorem~4.5.5]{HoughKrishnapurPeresVirag2009} gives \(0\le q\le1\) a.e.,
and comparing \(\mathbb EY([0,R])=\lambda R\) with the trace of the restricted
translation-invariant operator gives
\(\lambda=(2\pi)^{-1}\int_{\mathbb R}q(\xi)\,\mathrm d\xi\). Thus
\(q\in L^1(\mathbb R)\), and we may take
\[
\mathcal{K}(r)=(2\pi)^{-1}\int_{\mathbb{R}}e^{ir\xi}q(\xi)\,\mathrm{d}\xi .
\]
Then \(\mathcal K\) is continuous and \(\mathcal K(0)=\lambda\). For \(x\neq y\),
the determinantal two-point formula, stationarity, and Hermitian symmetry give
\(\rho_Y^{(2)}(r)=\lambda^2-|\mathcal K(r)|^2\), with \(r=x-y\). Letting
\(r\to0\) proves the claim.
\end{proof}

\begin{proposition}
\label{prop:not-translation-invariant-dpp}
Let \(I\subset(0,\infty)\) be admissible and have positive Lebesgue measure. Then \(\mathcal{X}_{I}\) admits no Hermitian translation-invariant locally trace-class determinantal representation on \(\mathbb{R}\) with respect to Lebesgue measure.
\end{proposition}

\begin{proof}
Assume that such a determinantal representation exists. Since \(\mathcal{X}_{I}\) has finite intensity \(\lambda_{I}\), Lemma~\ref{lem:translation-invariant-dpp-small-gap} applies and gives a continuous representative of the second intensity with \(\lim_{r\to0}\rho_{I}^{(2)}(r)=0\). Proposition~\ref{prop:small-gap-limit}, however, gives a continuous small-gap representative with 
\(\lim_{r\to0}\rho_I^{(2)}(r)>0.\)
These representatives describe the same second factorial moment measure, so equality almost everywhere, together with continuity, forces equality everywhere. This contradiction rules out such a representation.
\end{proof}

\section{Reduced cumulants and Brillinger mixing}
\label{sec:brillinger-mixing}

In this section we prove the all-order reduced factorial cumulant bound for the projected strip process. The proof combines Proposition~\ref{prop:projected-cycle-expansion} with a one-dimensional convolution identity for \(t\mapsto (t^{2}+\alpha^{2})^{-1}\). The resulting \(L^{1}\)-integrability gives Brillinger mixing and the linear cumulant estimates used later.

\subsection{Auxiliary integrability and convolution estimates}
\label{subsec:auxiliary-integrability-convolution}

The \(L^{1}\)-bound for the \(k\)-th reduced cumulant density will involve
\(\int_I y^{-(1+1/k)}\,\mathrm dy\). The following elementary consequence follows directly from admissibility and the decomposition
\(I=(I\cap(0,1])\cup(I\cap[1,\infty))\).

\begin{lemma}
\label{lem:admissibility-lp}
Let \(I\subset(0,\infty)\) be admissible. Then, for every \(1<p\leq2\),
\(\int_I y^{-p}\,\mathrm dy<\infty.\)
\end{lemma}

The following identity is the analytic core of the cumulant estimate. It evaluates the horizontal integral obtained by following one determinantal cycle.

\begin{lemma}
\label{lem:cycle-convolution-integral}
Let \(k\ge2\), \(\alpha_1,\ldots,\alpha_k>0\), and
\(g_\alpha(t):=(t^2+\alpha^2)^{-1}\). Then
\begin{equation}
\label{eq:cycle-convolution-integral}
\int_{\mathbb R^{k-1}}
g_{\alpha_1}(t_2)
g_{\alpha_2}(t_3-t_2)\cdots
g_{\alpha_{k-1}}(t_k-t_{k-1})
g_{\alpha_k}(t_k)
\,\mathrm dt_2\cdots\mathrm dt_k
=
\frac{\pi^{k-1}}{\alpha_1\cdots\alpha_k(\alpha_1+\cdots+\alpha_k)} .
\end{equation}
\end{lemma}

\begin{proof}
With \(s_2=t_2\), \(s_3=t_3-t_2,\ldots,s_k=t_k-t_{k-1}\), the Jacobian is \(1\) and
\(t_k=s_2+\cdots+s_k\). Since \(g_\alpha\) is even, the left-hand side equals
\((g_{\alpha_1}*\cdots*g_{\alpha_k})(0)\). Using
\(\widehat{g_\alpha}(\xi)=\pi\alpha^{-1}e^{-\alpha|\xi|}\), we get
\[
\widehat{g_{\alpha_1}*\cdots*g_{\alpha_k}}(\xi)
=
\frac{\pi^k}{\alpha_1\cdots\alpha_k}
e^{-(\alpha_1+\cdots+\alpha_k)|\xi|}.
\]
Fourier inversion at \(0\) gives \eqref{eq:cycle-convolution-integral}.
\end{proof}

\subsection{The Brillinger-mixing estimate}
\label{subsec:brillinger-mixing-estimate}

We now prove the Brillinger-mixing estimate stated in Theorem~\ref{thm:brillinger-mixing}. The proof gives the quantitative all-order \(L^{1}\)-bound for the reduced factorial cumulant densities, which is the main input for the linear cumulant estimates and the Gaussian limits.

\begin{proof}[Proof of Theorem~\ref{thm:brillinger-mixing}]
Assume that \(I\subset(0,\infty)\) is admissible with \(|I|>0\), and fix
\(k\ge2\). By \eqref{eq:reduced-cumulant-density-definition} and \eqref{eq:projected-cycle-expansion},
\[
\left\|c_I^{(k)}\right\|_{L^1(\mathbb R^{k-1})}
\le
\sum_{\sigma\in\operatorname{Cyc}_k}
\int_{I^k}A_\sigma(y_1,\ldots,y_k)\,\mathrm dy_1\cdots\mathrm dy_k,
\]
where
\[
A_\sigma(y_1,\ldots,y_k)
:=
\frac{1}{\pi^k}
\int_{\mathbb R^{k-1}}
\prod_{j=1}^k
\frac{1}{
(t_j-t_{\sigma(j)})^2+(y_j+y_{\sigma(j)})^2}
\,\mathrm dt_2\cdots\mathrm dt_k .
\]

Fix \(\sigma\in\operatorname{Cyc}_k\) and list the cycle as
\(i_1=1\), \(i_\ell=\sigma^{\ell-1}(1)\), \(2\le\ell\le k\), with
\(i_{k+1}=i_1\). The change of variables
\(s_1=t_{i_1}=0\), \(s_\ell=t_{i_\ell}\), \(2\le\ell\le k\), is a permutation of
\((t_2,\ldots,t_k)\). With \(\alpha_\ell:=y_{i_\ell}+y_{i_{\ell+1}}\),
Lemma~\ref{lem:cycle-convolution-integral} gives
\[
A_\sigma(y_1,\ldots,y_k)
=
\frac{1}{\pi\,\alpha_1\cdots\alpha_k(\alpha_1+\cdots+\alpha_k)}
\le
\frac{1}{2^{k+1}k\pi}
\prod_{j=1}^k y_j^{-(1+1/k)},
\]
where the last step uses the elementary lower bounds
\(\prod_\ell\alpha_\ell\ge2^k\prod_jy_j\) and
\(\sum_\ell\alpha_\ell\ge2k(\prod_jy_j)^{1/k}\). Since
\(|\operatorname{Cyc}_k|=(k-1)!\), this yields
\[
\left\|c_I^{(k)}\right\|_{L^1(\mathbb R^{k-1})}
\le
\frac{(k-1)!}{2^{k+1}k\pi}
\left(
\int_I y^{-(1+1/k)}\,\mathrm dy
\right)^k
<\infty
\]
by Lemma~\ref{lem:admissibility-lp}. Hence all reduced factorial cumulant
measures of order at least two have finite total variation, so
\(\mathcal X_I\) is Brillinger mixing.
\end{proof}

\subsection{Linear cumulant bounds for counts and statistics}
\label{subsec:linear-cumulant-bounds}

We now use the Brillinger-mixing estimate to derive the linear-in-volume cumulant bounds needed for the Gaussian limits. We first record an elementary multilinear estimate that converts \(L^{1}\)-control of reduced cumulant densities into volume bounds. We then apply it to factorial cumulants of interval counts, convert these estimates to ordinary cumulants, and finally treat compactly supported linear statistics.

\begin{lemma}
\label{lem:multilinear-cumulant-estimate}
Let \(k\geq2\), \(c^{(k)}\in L^{1}(\mathbb{R}^{k-1})\), and let \(1\leq p_{1},\ldots,p_{k}\leq\infty\) satisfy \(\sum_{j=1}^{k}p_j^{-1}=1\). If \(f_{j}\in L^{p_{j}}(\mathbb{R})\), \(1\leq j\leq k\), then
\begin{equation}
\label{eq:multilinear-cumulant-estimate}
\left|
\int_{\mathbb{R}^{k}}
\prod_{j=1}^{k}f_{j}(x_{j})
c^{(k)}(x_{2}-x_{1},\ldots,x_{k}-x_{1})
\,\mathrm{d}x_{1}\cdots\mathrm{d}x_{k}
\right|
\leq
\left\|c^{(k)}\right\|_{L^{1}(\mathbb{R}^{k-1})}
\prod_{j=1}^{k}\left\|f_{j}\right\|_{L^{p_{j}}(\mathbb{R})}.
\end{equation}
\end{lemma}

\begin{proof}
Let \(J\) be the left-hand side of \eqref{eq:multilinear-cumulant-estimate} with absolute values inside the integral. With \(t_j=x_j-x_1\), \(2\leq j\leq k\), Tonelli's theorem gives
\[
J
=
\int_{\mathbb{R}^{k-1}}
|c^{(k)}(t_{2},\ldots,t_{k})|
\int_{\mathbb{R}}
|f_{1}(x)|
\prod_{j=2}^{k}|f_{j}(x+t_{j})|
\,\mathrm{d}x\,
\mathrm{d}t_{2}\cdots\mathrm{d}t_{k}.
\]
For fixed \(t_{2},\ldots,t_{k}\), H\"older's inequality and translation invariance of Lebesgue measure bound the inner integral by \(\prod_{j=1}^{k}\|f_j\|_{L^{p_j}(\mathbb{R})}\). This proves \eqref{eq:multilinear-cumulant-estimate}. 
\end{proof}

\begin{corollary}
\label{cor:linear-factorial-cumulant-bounds}
Assume that \(I\subset(0,\infty)\) is admissible with \(|I|>0\). Then, for every
\(k\ge2\) and \(L\ge1\),
\begin{equation}
\label{eq:linear-factorial-cumulant-bound}
\left|\operatorname{cum}_{k}^{!}(N_I(L))\right|
\le
L\left\|c_I^{(k)}\right\|_{L^1(\mathbb R^{k-1})}.
\end{equation}
\end{corollary}

\begin{proof}
Let \(\varphi_L:=\mathbf 1_{[0,L]}\). Then
\(N_I(L)=\int_{\mathbb R}\varphi_L\,\mathrm d\mathcal X_I\). By
Corollary~\ref{cor:projected-reduced-factorial-cumulant-density} and
Theorem~\ref{thm:brillinger-mixing}, the reduced-cumulant representation is
absolutely convergent. Applying Lemma~\ref{lem:multilinear-cumulant-estimate}
with \(p_1=1\), \(p_2=\cdots=p_k=\infty\), and
\(f_1=\cdots=f_k=\varphi_L\), gives
\[
\left|\operatorname{cum}_{k}^{!}(N_I(L))\right|
\le
\|\varphi_L\|_{L^1}\|\varphi_L\|_{L^\infty}^{k-1}
\left\|c_I^{(k)}\right\|_{L^1(\mathbb R^{k-1})}
=
L\left\|c_I^{(k)}\right\|_{L^1(\mathbb R^{k-1})}.
\]
\end{proof}

\begin{corollary}
\label{cor:linear-ordinary-cumulant-bounds}
Assume that \(I\subset(0,\infty)\) is admissible with \(|I|>0\). For every
\(q\ge2\), there exists \(C'_{I,q}<\infty\) such that
\begin{equation}
\label{eq:linear-ordinary-cumulant-bound}
\left|\operatorname{cum}_{q}(N_I(L))\right|
\le
C'_{I,q}L,
\qquad L\ge1.
\end{equation}
\end{corollary}

\begin{proof}
Let \(\varphi_L:=\mathbf 1_{[0,L]}\). By
Lemma~\ref{lem:ordinary-cumulants-linear-statistics}, the ordinary cumulant is a
finite sum over partitions of \([q]\). The one-block partition contributes
\(\lambda_I L\). For a partition with \(m\ge2\) blocks, since
\(\varphi_L^{|B|}=\varphi_L\), the corresponding term is
\(\operatorname{cum}_{m}^{!}(N_I(L))\), whose absolute value is bounded by
\(L\|c_I^{(m)}\|_{L^1(\mathbb R^{m-1})}\) by
Corollary~\ref{cor:linear-factorial-cumulant-bounds}. Summing over the finitely
many partitions gives \eqref{eq:linear-ordinary-cumulant-bound}.
\end{proof}

For the white-noise limit, we also need the preceding estimates for compactly
supported test functions with macroscopically growing support.

\begin{corollary}
\label{cor:linear-statistic-cumulant-bounds}
Assume that \(I\subset(0,\infty)\) is admissible with \(|I|>0\). For every
\(q\ge2\), there exists \(C_{I,q}<\infty\) such that, for every bounded
measurable compactly supported \(\varphi:\mathbb R\to\mathbb R\),
\begin{equation}
\label{eq:compact-linear-statistic-cumulant-bound}
\left|
\operatorname{cum}_{q}
\left(
\int_{\mathbb R}\varphi(x)\,\mathcal X_I(\mathrm dx)
\right)
\right|
\le
C_{I,q}|\operatorname{supp}\varphi|
\max\left\{1,\|\varphi\|_{L^\infty(\mathbb R)}^q\right\}.
\end{equation}
\end{corollary}

\begin{proof}
By Lemma~\ref{lem:ordinary-cumulants-linear-statistics}, the ordinary cumulant
is a finite sum over partitions of \([q]\). The one-block term is bounded by
\(\lambda_I|\operatorname{supp}\varphi|\|\varphi\|_\infty^q\). For a partition
with \(m\ge2\) blocks, the corresponding term can be written using the reduced
density \(c_I^{(m)}\). Applying Lemma~\ref{lem:multilinear-cumulant-estimate}
with \(p_1=1\) and \(p_2=\cdots=p_m=\infty\) gives the bound
\[
|\operatorname{supp}\varphi|
\max\left\{1,\|\varphi\|_\infty^q\right\}
\|c_I^{(m)}\|_{L^1(\mathbb R^{m-1})}.
\]
Summing over partitions and using Theorem~\ref{thm:brillinger-mixing} proves
\eqref{eq:compact-linear-statistic-cumulant-bound}.
\end{proof}

\section{Gaussian limits}
\label{sec:gaussian-limits}

In this section we prove the macroscopic white-noise limit and the FCLT for
\(B_L(t):=(N_I(Lt)-\lambda_I Lt)/\sqrt L\), \(0\le t\le1\). The proof uses the
cumulant method: the linear-in-volume bounds from
Section~\ref{sec:brillinger-mixing} make all normalized cumulants of order at
least three vanish, while the second cumulant is identified through \(g_I\).

Throughout this section, \(I\subset(0,\infty)\) is admissible with \(|I|>0\).
Recall that
\(\lambda_I=\int_I(4\pi y^2)^{-1}\,\mathrm dy\) and
\(\sigma_I^2=\lambda_I+\int_{\mathbb R}g_I(r)\,\mathrm dr>0\).

\subsection{Variance of compactly supported linear statistics}
\label{subsec:variance-linear-statistics}

We first record the second-moment formula for compactly supported linear
statistics.

\begin{lemma}
\label{lem:variance-linear-statistics}
Let \(\varphi:\mathbb R\to\mathbb R\) be bounded, measurable, and compactly
supported, and set \(S(\varphi):=\int_{\mathbb R}\varphi\,\mathrm d\mathcal X_I\).
Then \(S(\varphi)\) has finite moments of all orders, and
\begin{equation}
\label{eq:variance-linear-statistic}
\operatorname{Var}(S(\varphi))
=
\lambda_I\int_{\mathbb R}\varphi(x)^2\,\mathrm dx
+
\int_{\mathbb R}\int_{\mathbb R}
\varphi(x)\varphi(y)g_I(y-x)
\,\mathrm dx\,\mathrm dy .
\end{equation}
\end{lemma}

\begin{proof}
Let \(K:=\operatorname{supp}\varphi\). For \(m\ge1\),
Corollary~\ref{cor:projected-intensity-basic-properties} gives
\[
\begin{aligned}
\mathbb E\bigl[
\mathcal X_I(K)(\mathcal X_I(K)-1)\cdots(\mathcal X_I(K)-m+1)
\bigr]
&=
\int_{K^m}\rho_I^{(m)}(x_1,\ldots,x_m)\,\mathrm dx_1\cdots\mathrm dx_m\\
&\le
\lambda_I^m|K|^m<\infty .
\end{aligned}
\]
Thus \(\mathcal X_I(K)\), and hence \(S(\varphi)\), has finite moments of all
orders. By simplicity and the factorial moment formula,
\[
\mathbb E[S(\varphi)^2]
=
\lambda_I\int_{\mathbb R}\varphi(x)^2\,\mathrm dx
+
\int_{\mathbb R}\int_{\mathbb R}
\varphi(x)\varphi(y)\rho_I^{(2)}(x,y)\,\mathrm dx\,\mathrm dy .
\]
Subtracting \(\mathbb E[S(\varphi)]^2\) and using
\(\rho_I^{(2)}(x,y)-\lambda_I^2=g_I(y-x)\) gives
\eqref{eq:variance-linear-statistic}. The second integral is absolutely
convergent since \(\varphi\) is bounded and compactly supported and
\(g_I\in L^1(\mathbb R)\).
\end{proof}

\subsection{White-noise finite-dimensional limit}
\label{subsec:white-noise-limit}

The next lemma identifies the macroscopic covariance contribution from an
\(L^1\) kernel.

\begin{lemma}
\label{lem:macroscopic-covariance-limit}
Let \(g\in L^1(\mathbb R)\), and let \(f,h\in L^\infty(\mathbb R)\) be compactly
supported. Set \(f_L(x):=f(x/L)\) and \(h_L(x):=h(x/L)\). Then
\begin{equation}
\label{eq:macroscopic-covariance-limit}
\lim_{L\to\infty}
\frac{1}{L}
\int_{\mathbb R}\int_{\mathbb R}
f_L(x)h_L(y)g(y-x)\,\mathrm dx\,\mathrm dy
=
\left(\int_{\mathbb R}g(r)\,\mathrm dr\right)
\left(\int_{\mathbb R}f(u)h(u)\,\mathrm du\right).
\end{equation}
\end{lemma}

\begin{proof}
The normalized integral equals, after the changes of variables \(r=y-x\) and
\(u=x/L\),
\[
\int_{\mathbb R}g(r)A_L(r)\,\mathrm dr,
\qquad
A_L(r):=\int_{\mathbb R}f(u)h(u+r/L)\,\mathrm du .
\]
For fixed \(r\), translation continuity in \(L^1\) gives
\(A_L(r)\to\int_{\mathbb R}f(u)h(u)\,\mathrm du\). Since \(A_L\) is uniformly
bounded and \(g\in L^1(\mathbb R)\), dominated convergence proves
\eqref{eq:macroscopic-covariance-limit}.
\end{proof}

For \(L\ge1\) and \(\varphi\in C_c^\infty(\mathbb R)\), recall
\[
\Xi_L(\varphi)
:=
\frac{1}{\sqrt L}
\left(
\int_{\mathbb R}\varphi(x/L)\,\mathcal X_I(\mathrm dx)
-
\lambda_I L\int_{\mathbb R}\varphi(u)\,\mathrm du
\right).
\]

\begin{proof}[Proof of Theorem~\ref{thm:white-noise-limit}]
Fix \(\varphi\in C_c^\infty(\mathbb R)\), set
\(\varphi_L(x):=\varphi(x/L)\), and
\(S_L:=\int_{\mathbb R}\varphi_L\,\mathrm d\mathcal X_I\). Then
\(\Xi_L(\varphi)=(S_L-\mathbb E S_L)/\sqrt L\). By
Lemma~\ref{lem:variance-linear-statistics} and
Lemma~\ref{lem:macroscopic-covariance-limit},
\[
\lim_{L\to\infty}\operatorname{Var}(\Xi_L(\varphi))
=
\left(\lambda_I+\int_{\mathbb R}g_I(r)\,\mathrm dr\right)
\int_{\mathbb R}\varphi(u)^2\,\mathrm du
=
\sigma_I^2\int_{\mathbb R}\varphi(u)^2\,\mathrm du .
\]

For \(q\ge3\), centering does not affect the \(q\)-th cumulant. Applying
Corollary~\ref{cor:linear-statistic-cumulant-bounds} to \(\varphi_L\) gives
\[
\left|\operatorname{cum}_q(\Xi_L(\varphi))\right|
=
L^{-q/2}\left|\operatorname{cum}_q(S_L)\right|
\le
C_{I,q,\varphi}L^{1-q/2}
\to0 .
\]
Thus the cumulants converge to those of the centered Gaussian with variance
\(\sigma_I^2\int_{\mathbb R}\varphi^2\). By the moment--cumulant formula
\eqref{eq:appendix-moment-cumulant-inversion} and moment determinacy of the
Gaussian law,
\[
\Xi_L(\varphi)
\Rightarrow
\mathcal N\left(0,\sigma_I^2\int_{\mathbb R}\varphi(u)^2\,\mathrm du\right).
\]

For joint convergence, let \(a_1,\ldots,a_m\in\mathbb R\) and
\(\varphi_1,\ldots,\varphi_m\in C_c^\infty(\mathbb R)\). By linearity,
\[
\sum_{j=1}^m a_j\Xi_L(\varphi_j)
=
\Xi_L\left(\sum_{j=1}^m a_j\varphi_j\right).
\]
The one-dimensional convergence applied to \(\sum_j a_j\varphi_j\), together
with the Cram\'er--Wold theorem \cite[Corollary~6.5]{Kallenberg2021}, gives the
claimed centered Gaussian vector; the covariance is obtained by polarization.
\end{proof}

The proof of Theorem~\ref{thm:white-noise-limit} uses only boundedness, compact support, and Lemma~\ref{lem:macroscopic-covariance-limit}; hence it also gives the following extension.

\begin{corollary}
\label{cor:white-noise-bounded-test-functions}
The finite-dimensional convergence in Theorem~\ref{thm:white-noise-limit} remains valid for bounded measurable compactly supported test functions.
\end{corollary}

\subsection{Finite-dimensional convergence of the counting path}
\label{subsec:fclt-finite-dimensional}

The finite-dimensional convergence follows from the white-noise limit applied to
indicator functions.

\begin{lemma}
\label{lem:fclt-finite-dimensional-convergence}
For every \(m\ge1\) and \(0\le t_1<\cdots<t_m\le1\),
\[
(B_L(t_1),\ldots,B_L(t_m))
\Rightarrow
(B(t_1),\ldots,B(t_m)),
\]
where \(B\) is Brownian motion with covariance
\(\mathbb E[B(s)B(t)]=\sigma_I^2\min\{s,t\}\).
\end{lemma}

\begin{proof}
Fix \(a_1,\ldots,a_m\in\mathbb R\) and set
\(\psi:=\sum_{j=1}^m a_j\mathbf 1_{[0,t_j]}\). Then
\[
\sum_{j=1}^m a_jB_L(t_j)
=
\frac{1}{\sqrt L}
\left(
\int_{\mathbb R}\psi(x/L)\,\mathcal X_I(\mathrm dx)
-
\lambda_I L\int_{\mathbb R}\psi(u)\,\mathrm du
\right).
\]
By Corollary~\ref{cor:white-noise-bounded-test-functions}, this converges to a centered
Gaussian variable with variance
\[
\sigma_I^2\int_{\mathbb R}\psi(u)^2\,\mathrm du
=
\sigma_I^2\sum_{i,j=1}^m a_ia_j\min\{t_i,t_j\}.
\]
The Cram\'er--Wold theorem \cite[Corollary~6.5]{Kallenberg2021} gives the
claimed finite-dimensional convergence.
\end{proof}

\subsection{Tightness of polygonal interpolations}
\label{subsec:fclt-tightness}

It remains to prove tightness in \(D([0,1])\). We compare \(B_L\) with the
polygonal interpolation of the centered count process sampled at integer times.
Set \(X(u):=N_I(u)-\lambda_Iu\), \(u\ge0\), and, for \(L\ge1\) and \(0\le t\le1\),
\[
\widetilde B_L(t)
:=
\frac{1}{\sqrt L}
\left[
X(\lfloor Lt\rfloor)
+
(Lt-\lfloor Lt\rfloor)
\bigl(X(\lfloor Lt\rfloor+1)-X(\lfloor Lt\rfloor)\bigr)
\right],
\]
Then \(\widetilde B_L\in C([0,1])\) and \(\widetilde B_L(0)=0\).

\begin{lemma}
\label{lem:fourth-moment-increment-bound}
There exists \(C<\infty\), depending only on \(I\), such that
\(\mathbb E[|X(v)-X(u)|^4]\le C(1+|v-u|)^2\) for all \(0\le u\le v\).
\end{lemma}

\begin{proof}
By stationarity and since fixed points are not charged, \(X(v)-X(u)\) has the
same law as \(N_I(v-u)-\lambda_I(v-u)\). Let \(\ell:=v-u\). If \(\ell<1\), then
\(|X(v)-X(u)|\le \mathcal X_I([u,u+1])+\lambda_I\), whose fourth moment is
finite by stationarity and the factorial moment bounds used in
Lemma~\ref{lem:variance-linear-statistics}.

For \(\ell\ge1\), Corollary~\ref{cor:linear-ordinary-cumulant-bounds} gives
\(|\operatorname{cum}_2(N_I(\ell))|+|\operatorname{cum}_4(N_I(\ell))|\le C_1\ell\).
The same bounds hold for \(Y:=N_I(\ell)-\lambda_I\ell\). Since \(Y\) is centered,
\eqref{eq:appendix-fourth-moment-centered} gives
\(\mathbb E[|Y|^4]=\operatorname{cum}_4(Y)+3\operatorname{cum}_2(Y)^2\le C_2\ell^2\).
The two cases prove the claim.
\end{proof}

\begin{lemma}
\label{lem:polygonal-tightness}
The family \((\widetilde B_L)_{L\ge1}\) is tight in \(D([0,1])\).
\end{lemma}

\begin{proof}
Since each \(\widetilde B_L\) is continuous and \(\widetilde B_L(0)=0\), the
Kolmogorov--Chentsov tightness criterion
(Lemma~\ref{lem:kolmogorov-chentsov-tightness}), together with the embedding
\(C([0,1])\hookrightarrow D([0,1])\), reduces the proof to
\(\mathbb E[|\widetilde B_L(t)-\widetilde B_L(s)|^4]\le C|t-s|^2\) for
\(0\le s<t\le1\). Let \(a:=Ls\), \(b:=Lt\), and let \(\widetilde X\) be the
polygonal interpolation of \((X(n))_{n\ge0}\). Then
\[
\widetilde B_L(t)-\widetilde B_L(s)
=
L^{-1/2}\bigl(\widetilde X(b)-\widetilde X(a)\bigr).
\]

If \(b-a\ge1\), set \(k:=\lfloor a\rfloor\) and \(\ell:=\lfloor b\rfloor\). Then
\[
|\widetilde X(b)-\widetilde X(a)|
\le
|X(\ell)-X(k)|+|X(k+1)-X(k)|+|X(\ell+1)-X(\ell)|.
\]
Lemma~\ref{lem:fourth-moment-increment-bound} gives
\[
\mathbb E[|\widetilde B_L(t)-\widetilde B_L(s)|^4]
\le
C L^{-2}\bigl((1+\ell-k)^2+1\bigr)
\le
C|t-s|^2,
\]
since \(\ell-k\le L(t-s)+1\) and \(L(t-s)\ge1\).

If \(b-a<1\), then \([a,b]\) meets at most two unit mesh intervals, so
\[
|\widetilde X(b)-\widetilde X(a)|
\le
(b-a)\bigl(|\Delta X(n_1)|+|\Delta X(n_2)|\bigr),
\qquad
\Delta X(n):=X(n+1)-X(n),
\]
with one term omitted if necessary. By stationarity and
Lemma~\ref{lem:fourth-moment-increment-bound}, \(\Delta X(n)\) has uniformly
bounded fourth moment. Hence
\[
\mathbb E[|\widetilde B_L(t)-\widetilde B_L(s)|^4]
\le
C L^2(t-s)^4
\le
C|t-s|^2,
\]
because \(L(t-s)<1\). This proves the required Kolmogorov--Chentsov bound, hence
tightness in \(D([0,1])\).
\end{proof}

\subsection{Tightness transfer and completion of the FCLT}
\label{subsec:proof-fclt}

It remains to transfer tightness from the polygonal interpolation to the original
counting path.

\begin{lemma}
\label{lem:polygonal-approximation}
As \(L\to\infty\),
\(\sup_{0\le t\le1}|B_L(t)-\widetilde B_L(t)|\to0\) in probability.
\end{lemma}

\begin{proof}
Fix \(t\in[0,1]\) and put \(k:=\lfloor Lt\rfloor\). Since \(Lt\in[k,k+1]\),
\(N_I\) is nondecreasing, and \(\widetilde B_L(t)\) is a convex combination of
\(X(k)/\sqrt L\) and \(X(k+1)/\sqrt L\),
\[
\sup_{0\le t\le1}|B_L(t)-\widetilde B_L(t)|
\le
\frac{2}{\sqrt L}
\max_{0\le k\le\lceil L\rceil}
\bigl(\mathcal X_I([k,k+1])+\lambda_I\bigr).
\]
By stationarity, the finite fourth moment of \(\mathcal X_I([0,1])\), and a
union bound, for every \(\varepsilon>0\),
\[
\mathbb P\left(
\max_{0\le k\le\lceil L\rceil}\mathcal X_I([k,k+1])
>\varepsilon\sqrt L
\right)
\le
(\lceil L\rceil+1)
\frac{\mathbb E[\mathcal X_I([0,1])^4]}{\varepsilon^4L^2}
\to0 .
\]
The deterministic \(\lambda_I/\sqrt L\) term is negligible, proving the claim.
\end{proof}

\begin{lemma}
\label{lem:fclt-tightness}
The family \((B_L)_{L\ge1}\) is asymptotically tight in \(D([0,1])\): for every
sequence \(L_n\to\infty\), \((B_{L_n})_{n\ge1}\) is tight.
\end{lemma}

\begin{proof}
Let \(L_n\to\infty\). By Lemma~\ref{lem:polygonal-tightness},
\((\widetilde B_{L_n})_{n\ge1}\) is tight in \(D([0,1])\). Moreover, by Lemma~\ref{lem:polygonal-approximation} and the fact that the
Skorokhod \(J_1\) metric \(d_{J_1}\) is controlled by the uniform distance, as
recalled in Section~\ref{subsec:appendix-skorokhod-space},
\[
d_{J_1}(B_{L_n},\widetilde B_{L_n})
\le
\sup_{0\le t\le1}|B_{L_n}(t)-\widetilde B_{L_n}(t)|
\to0
\]
in probability. Given any subsequence, choose a further subsequence along which
\(\widetilde B_{L_n}\Rightarrow Y\). Then Billingsley's convergence theorem
\cite[Theorem~3.1]{Billingsley1999} gives \(B_{L_n}\Rightarrow Y\) along the
same further subsequence. Hence every subsequence has a weakly convergent
subsequence, so \((B_{L_n})_{n\ge1}\) is tight.
\end{proof}

\begin{proof}[Proof of Theorem~\ref{thm:fclt}]
Let \(L_n\to\infty\). By Lemma~\ref{lem:fclt-tightness},
\((B_{L_n})_{n\ge1}\) is tight in \(D([0,1])\). By
Lemma~\ref{lem:fclt-finite-dimensional-convergence}, its finite-dimensional
distributions converge to those of Brownian motion \(B\) with covariance
\(\mathbb E[B(s)B(t)]=\sigma_I^2\min\{s,t\}\). Since \(B\) has continuous sample
paths, tightness in \(D([0,1])\) and convergence of finite-dimensional
distributions imply \(B_{L_n}\Rightarrow B\) in the Skorokhod \(J_1\) topology
(see Lemma~\ref{lem:tightness-plus-finite-dimensional-convergence}, or
Billingsley~\cite[Section~13]{Billingsley1999}). Since \(L_n\to\infty\) was
arbitrary, \(B_L\Rightarrow B\) in \(D([0,1])\).
\end{proof}

\begin{proof}[Proof of Corollary~\ref{cor:interval-count-clt}]
The evaluation map \(x\mapsto x(1)\) is continuous at every \(x\in D([0,1])\)
that is continuous at \(1\). Since Brownian motion has continuous sample paths
almost surely, Theorem~\ref{thm:fclt} and the continuous mapping theorem give
\((N_I(L)-\lambda_I L)/\sqrt L=B_L(1)\Rightarrow B(1)\). As
\(B(1)\sim\mathcal N(0,\sigma_I^2)\), the result follows.
\end{proof}

\begin{appendix}
\section{Point-process cumulants}
\label{app:cumulants}

This appendix records the point-process conventions used in the main text and separates the general cumulant formalism from the model-specific analysis of the projected Hardy--Szeg\H{o} process. All point processes considered here are defined on \(\mathbb{R}\), unless stated otherwise; the same definitions apply, with the evident changes, to point processes on \(\mathbb{H}\) with respect to planar Lebesgue measure \(\mathrm{d}A\). For background on locally finite point processes, factorial moment measures, and point-process conventions, see
\cite{DaleyVereJones2003,MollerWaagepetersen2004,DaleyVereJones2008,
LastPenrose2018}.

\subsection{Factorial moment measures and ordinary cumulants}
\label{subsec:appendix-factorial-moments-ordinary-cumulants}

Let \(\mathcal{X}\) be a locally finite point process on \(\mathbb{R}\). We call it simple if \(\mathcal{X}(\{x\})\in\{0,1\}\) for every \(x\in\mathbb{R}\), almost surely, and stationary if, for every \(a\in\mathbb{R}\), the translated process \(B\mapsto\mathcal{X}(B+a)\) has the same distribution as \(\mathcal{X}\). All moment and cumulant measures below are understood in the boundedly finite sense, that is, finite on bounded Borel sets.

\begin{definition}
\label{def:appendix-factorial-moment-measures}
Let \(\mathcal{X}\) be a point process on \(\mathbb{R}\). For \(k\geq1\), the \(k\)-th factorial moment measure \(\alpha^{(k)}\) is defined on bounded Borel rectangles by
\[
\alpha^{(k)}(B_{1}\times\cdots\times B_{k})
:=
\mathbb{E}
\left[
\sum_{x_{1},\ldots,x_{k}\in\mathcal{X}}^{\neq}
\mathbf{1}_{B_{1}}(x_{1})\cdots\mathbf{1}_{B_{k}}(x_{k})
\right],
\]
where the summation is over ordered \(k\)-tuples of pairwise distinct points of \(\mathcal{X}\). The measure \(\alpha^{(1)}\) is the intensity measure.
\end{definition}

\begin{definition}
\label{def:appendix-correlation-functions}
Let \(k\geq1\). A point process \(\mathcal{X}\) admits a \(k\)-point intensity function, or \(k\)-point correlation function, if \(\alpha^{(k)}\) is absolutely continuous with respect to Lebesgue measure on \(\mathbb{R}^{k}\). Its density is denoted by \(\rho^{(k)}\), so
\(
\alpha^{(k)}(\mathrm{d}x_{1}\cdots\mathrm{d}x_{k})
=
\rho^{(k)}(x_{1},\ldots,x_{k})
\,\mathrm{d}x_{1}\cdots\mathrm{d}x_{k}.
\)
Equivalently, for every nonnegative measurable \(F:\mathbb{R}^{k}\to[0,\infty]\),
\[
\mathbb{E}
\left[
\sum_{x_{1},\ldots,x_{k}\in\mathcal{X}}^{\neq}
F(x_{1},\ldots,x_{k})
\right]
=
\int_{\mathbb{R}^{k}}
F(x_{1},\ldots,x_{k})
\rho^{(k)}(x_{1},\ldots,x_{k})
\,\mathrm{d}x_{1}\cdots\mathrm{d}x_{k}.
\]
\end{definition}

\begin{definition}
\label{def:appendix-ordinary-cumulants}
Let \(Y_{1},\ldots,Y_{k}\) be real-valued random variables with finite moments up to order \(k\). Their joint cumulant is
\begin{equation}
\label{eq:appendix-cumulant-partition-formula}
\operatorname{cum}(Y_{1},\ldots,Y_{k})
:=
\sum_{\pi\in\mathcal{P}([k])}
(-1)^{|\pi|-1}(|\pi|-1)!
\prod_{B\in\pi}
\mathbb{E}
\left[
\prod_{j\in B}Y_{j}
\right],
\end{equation}
where \([k]:=\{1,\ldots,k\}\), \(\mathcal{P}([k])\) is the set of all partitions of \([k]\), and \(|\pi|\) is the number of blocks of \(\pi\). For one random variable \(Y\), write \(\operatorname{cum}_{k}(Y):=\operatorname{cum}(Y,\ldots,Y)\).
\end{definition}

Moments are recovered from cumulants by
\begin{equation}
\label{eq:appendix-moment-cumulant-inversion}
\mathbb{E}\left[\prod_{j=1}^{k}Y_{j}\right]
=
\sum_{\pi\in\mathcal{P}([k])}
\prod_{B\in\pi}
\operatorname{cum}\left((Y_{j})_{j\in B}\right).
\end{equation}
In particular, 
\(\operatorname{cum}_1(Y)=\mathbb{E}[Y]\) and \( \operatorname{cum}_2(Y)=\operatorname{Var}(Y)\). If \(\mathbb{E}[Y]=0\), we have
\begin{equation}
\label{eq:appendix-fourth-moment-centered}
\mathbb{E}[Y^{4}]
=
\operatorname{cum}_{4}(Y)+3\operatorname{cum}_{2}(Y)^{2}.
\end{equation}

\subsection{Factorial cumulants and cumulant densities}
\label{subsec:appendix-factorial-cumulants-densities}

\begin{definition}
\label{def:appendix-factorial-cumulant-measures}
Let \(\mathcal{X}\) be a point process on \(\mathbb{R}\), and suppose that its factorial moment measures \(\alpha^{(j)}\), \(1\leq j\leq k\), exist as boundedly finite measures. The \(k\)-th factorial cumulant measure \(\gamma^{(k)}\) is defined on bounded Borel rectangles by
\[
\gamma^{(k)}(A_{1}\times\cdots\times A_{k})
=
\sum_{\pi\in\mathcal{P}([k])}
(-1)^{|\pi|-1}(|\pi|-1)!
\prod_{B\in\pi}
\alpha^{(|B|)}(A_{B}),
\]
where \(A_{B}:=\prod_{j\in B}A_{j}\), with the coordinates ordered increasingly.
\end{definition}

\begin{definition}
\label{def:appendix-factorial-cumulants-linear-statistics}
Let \(\mathcal{X}\) have \(k\)-th factorial cumulant measure \(\gamma^{(k)}\). For bounded measurable compactly supported \(\varphi_{1},\ldots,\varphi_{k}:\mathbb{R}\to\mathbb{R}\), define
\[
\operatorname{cum}_{k}^{!}
\left(
\int_{\mathbb{R}}\varphi_{1}\,\mathrm{d}\mathcal{X},
\ldots,
\int_{\mathbb{R}}\varphi_{k}\,\mathrm{d}\mathcal{X}
\right)
:=
\int_{\mathbb{R}^{k}}
\varphi_{1}(x_{1})\cdots\varphi_{k}(x_{k})
\,\gamma^{(k)}(\mathrm{d}x_{1}\cdots\mathrm{d}x_{k}).
\]
When \(\varphi_{1}=\cdots=\varphi_{k}=\varphi\), we write this as
\(\operatorname{cum}_k^! \left(\int_{\mathbb{R}}\varphi\,\mathrm{d}\mathcal{X}\right).\)
\end{definition}

We use \(\operatorname{cum}_{k}\) for ordinary cumulants of random variables and \(\operatorname{cum}_{k}^{!}\) for factorial cumulants associated with factorial cumulant measures. The following standard identity follows from \eqref{eq:appendix-cumulant-partition-formula} and Definition~\ref{def:appendix-factorial-cumulant-measures} by the usual grouping of coincident point variables in products of linear statistics.

\begin{lemma}
\label{lem:ordinary-cumulants-linear-statistics}
Let \(\mathcal X\) be a simple point process on \(\mathbb R\) with factorial
moment measures \(\alpha^{(m)}\) and factorial cumulant measures \(\gamma^{(m)}\),
\(1\le m\le q\). Let \(\psi_1,\ldots,\psi_q\) be bounded measurable functions
with compact support, and set
\(Y_j:=\int_{\mathbb R}\psi_j\,\mathrm d\mathcal X\). Then
\begin{equation}
\label{eq:ordinary-cumulants-linear-statistics}
\operatorname{cum}(Y_1,\ldots,Y_q)
=
\sum_{\pi=\{B_1,\ldots,B_{|\pi|}\}\in\mathcal P([q])}
\int_{\mathbb R^{|\pi|}}
\prod_{\ell=1}^{|\pi|}
\prod_{j\in B_\ell}\psi_j(x_\ell)
\,\gamma^{(|\pi|)}(\mathrm dx_1\cdots\mathrm dx_{|\pi|}).
\end{equation}
In particular, for every bounded measurable compactly supported \(\varphi\),
\begin{equation}
\label{eq:ordinary-cumulants-one-linear-statistic}
\operatorname{cum}_q
\left(
\int_{\mathbb R}\varphi\,\mathrm d\mathcal X
\right)
=
\sum_{\pi=\{B_1,\ldots,B_{|\pi|}\}\in\mathcal P([q])}
\int_{\mathbb R^{|\pi|}}
\prod_{\ell=1}^{|\pi|}
\varphi(x_\ell)^{|B_\ell|}
\,\gamma^{(|\pi|)}(\mathrm dx_1\cdots\mathrm dx_{|\pi|}).
\end{equation}
\end{lemma}

\begin{definition}
\label{def:appendix-factorial-cumulant-densities}
Let \(\mathcal{X}\) be a simple point process on \(\mathbb{R}\), and assume that its correlation functions \(\rho^{(m)}\) exist and are locally integrable for \(1\leq m\leq k\). The \(k\)-th factorial cumulant density \(u^{(k)}\) is defined, for pairwise distinct \(x_{1},\ldots,x_{k}\in\mathbb{R}\), by
\begin{equation}
\label{eq:appendix-factorial-cumulant-density}
u^{(k)}(x_{1},\ldots,x_{k})
=
\sum_{\pi\in\mathcal{P}([k])}
(-1)^{|\pi|-1}(|\pi|-1)!
\prod_{B\in\pi}
\rho^{(|B|)}(x_{B}),
\end{equation}
where \(x_{B}\) denotes the tuple \((x_{j})_{j\in B}\). The values of \(u^{(k)}\) on diagonals may be chosen arbitrarily.
\end{definition}

The inverse relation is the cluster expansion
\begin{equation}
\label{eq:appendix-cluster-expansion}
\rho^{(k)}(x_{1},\ldots,x_{k})
=
\sum_{\pi\in\mathcal{P}([k])}
\prod_{B\in\pi}
u^{(|B|)}(x_{B}).
\end{equation}
The two formulas \eqref{eq:appendix-factorial-cumulant-density} and \eqref{eq:appendix-cluster-expansion} are related by M\"obius inversion on the lattice of set partitions.

The next identity follows by comparing the permutation expansion of determinantal correlation functions with the cluster expansion \eqref{eq:appendix-cluster-expansion}.

\begin{lemma}
\label{lem:determinantal-cycle-expansion}
Let \(\mathcal X\) be a determinantal point process with kernel \(K\). Then, for
\(k\ge2\), the \(k\)-th factorial cumulant density is, off the diagonals,
\[
u^{(k)}(x_1,\ldots,x_k)
=
(-1)^{k-1}
\sum_{\sigma\in\operatorname{Cyc}_k}
\prod_{j=1}^k K(x_j,x_{\sigma(j)}),
\]
where \(\operatorname{Cyc}_k\) is the set of all \(k\)-cycles on
\(\{1,\ldots,k\}\).
\end{lemma}

\subsection{Reduced cumulants and Brillinger mixing}
\label{subsec:appendix-reduced-cumulants-brillinger-mixing}

\begin{definition}
\label{def:appendix-reduced-factorial-cumulant-measures}
Let \(\mathcal X\) be a stationary point process on \(\mathbb R\), and let
\(\gamma^{(k)}\) be its \(k\)-th factorial cumulant measure. The \(k\)-th reduced
factorial cumulant measure \(\gamma_{\operatorname{red}}^{(k)}\) is determined by
\begin{equation}
\label{eq:appendix-reduced-factorial-cumulant-measure}
\int_{\mathbb R^k}F(x_1,\ldots,x_k)\,
\gamma^{(k)}(\mathrm dx_1\cdots\mathrm dx_k)
=
\int_{\mathbb R}\int_{\mathbb R^{k-1}}
F(x,x+t_2,\ldots,x+t_k)\,
\gamma_{\operatorname{red}}^{(k)}(\mathrm dt_2\cdots\mathrm dt_k)\,\mathrm dx
\end{equation}
for every bounded measurable compactly supported \(F:\mathbb R^k\to\mathbb R\).
\end{definition}

\begin{definition}
\label{def:appendix-reduced-factorial-cumulant-densities}
If \(\gamma_{\operatorname{red}}^{(k)}\) is absolutely continuous with respect to
Lebesgue measure on \(\mathbb R^{k-1}\), its density \(c^{(k)}\), defined by
\[
\gamma_{\operatorname{red}}^{(k)}(\mathrm dt_2\cdots\mathrm dt_k)
=c^{(k)}(t_2,\ldots,t_k)\,\mathrm dt_2\cdots\mathrm dt_k,
\]
is called the
\(k\)-th reduced factorial cumulant density.
\end{definition}

\begin{definition}[Brillinger mixing]
\label{def:appendix-brillinger-mixing}
A stationary point process \(\mathcal X\) on \(\mathbb R\) is Brillinger mixing
if \(|\gamma_{\operatorname{red}}^{(k)}|(\mathbb R^{k-1})<\infty\) for every
\(k\ge2\).
\end{definition}

Applying Definition~\ref{def:appendix-factorial-cumulants-linear-statistics} together with the reduced disintegration formula \eqref{eq:appendix-reduced-factorial-cumulant-measure} gives the following representation.

\begin{lemma}
\label{lem:appendix-reduced-cumulant-linear-statistic}
Let \(\mathcal X\) be stationary, and suppose that its \(k\)-th reduced factorial
cumulant measure has density \(c^{(k)}\). Let \(\varphi:\mathbb R\to\mathbb R\)
be bounded, measurable, and compactly supported. If the integral below is
absolutely convergent, then
\[
\operatorname{cum}_k^!\left(\int_{\mathbb R}\varphi\,\mathrm d\mathcal X\right)
=
\int_{\mathbb R^k}
\varphi(x_1)\cdots\varphi(x_k)
c^{(k)}(x_2-x_1,\ldots,x_k-x_1)
\,\mathrm dx_1\cdots\mathrm dx_k .
\]
\end{lemma}

\section{Tightness tools}
\label{app:tightness}

This appendix records the Skorokhod-space and tightness facts used in the proof
of the functional central limit theorem. We include only the results needed in
the main text. Standard references include
\cite{Billingsley1999,Whitt2002,JacodShiryaev2003,Kallenberg2021}.

\subsection{The Skorokhod space}
\label{subsec:appendix-skorokhod-space}

Let \(D([0,1])\) denote the space of real-valued c\`adl\`ag functions on
\([0,1]\), and let \(\Lambda\) be the set of strictly increasing continuous
bijections of \([0,1]\). We use the Skorokhod \(J_1\) topology induced by
\begin{equation}
\label{eq:appendix-skorokhod-j1-metric}
d_{J_1}(x,y)
:=
\inf_{\lambda\in\Lambda}
\left\{
\sup_{0\le t\le1}|\lambda(t)-t|
\vee
\sup_{0\le t\le1}|x(t)-y(\lambda(t))|
\right\}.
\end{equation}
This metric is used only to describe the \(J_1\) topology. Although it is not
complete in general, the same topology is induced by equivalent complete
metrics, so \(D([0,1])\) is Polish. All tightness and weak-convergence statements
below refer to this topology.

By taking the identity time change in \eqref{eq:appendix-skorokhod-j1-metric}, we have
\(d_{J_1}(x,y)\le \sup_{0\le t\le1}|x(t)-y(t)|\) for \(x,y\in D([0,1])\).
Consequently, uniform convergence on \(C([0,1])\) implies convergence in the Skorokhod \(J_1\) topology, so the inclusion \(C([0,1])\hookrightarrow D([0,1])\) is continuous.

\subsection{Tightness and convergence criteria}
\label{subsec:appendix-tightness-criteria}

\begin{definition}
\label{def:appendix-tightness}
Let \(S\) be a metric space with its Borel \(\sigma\)-field. A family of
\(S\)-valued random elements \((Y_\alpha)_{\alpha\in A}\) is tight if, for every
\(\varepsilon>0\), there exists a compact set \(K\subset S\) such that
\(\sup_{\alpha\in A}\mathbb P(Y_\alpha\notin K)<\varepsilon\).
\end{definition}

\begin{lemma}[Kallenberg~{\cite[Theorem~23.7]{Kallenberg2021}}]
\label{lem:kolmogorov-chentsov-tightness}
Let \((Y_n)_{n\ge1}\) be a family of continuous processes on \([0,1]\). Suppose
that \(Y_n(0)=0\) for all \(n\), and that there exist constants \(C<\infty\),
\(\alpha>0\), and \(\beta>1\) such that
\(\mathbb E[|Y_n(t)-Y_n(s)|^\alpha]\le C|t-s|^\beta\) for all
\(0\le s<t\le1\) and all \(n\). Then \((Y_n)_{n\ge1}\) is tight in
\(C([0,1])\), equipped with the supremum norm.
\end{lemma}

\begin{lemma}[Billingsley~{\cite[Section~13]{Billingsley1999}}]
\label{lem:tightness-plus-finite-dimensional-convergence}
Let \((Y_n)_{n\ge1}\) be a tight family of \(D([0,1])\)-valued random elements.
Suppose that the finite-dimensional distributions of \(Y_n\) converge to those
of a process \(Y\), and that \(Y\) has continuous sample paths almost surely.
Then \(Y_n\Rightarrow Y\) in \(D([0,1])\), equipped with the Skorokhod \(J_1\)
topology.
\end{lemma}

\end{appendix}

\section*{Acknowledgments}
The authors would like to thank Xiang Fang and Shengzhao Hou for helpful discussions, valuable insights, and constructive comments, which have contributed to the improvement of the present work. The authors are also grateful for their encouragement and continued interest in this project.

The authors used ChatGPT to assist with language editing, LaTeX formatting, and exposition refinement; all AI-assisted material was verified by the authors, who take full responsibility for the final manuscript.

\end{document}